\documentclass[11pt]{amsart}

\usepackage{amsmath}               
\usepackage{amsfonts}       
\usepackage{amsthm}                
\usepackage{amssymb}
\usepackage{setspace}
\usepackage{verbatim}

\numberwithin{equation}{section}

\theoremstyle{plain}
\newtheorem{thm}[equation]{Theorem}
\newtheorem{lemma}[equation]{Lemma}
\newtheorem{proposition}[equation]{Proposition}

\theoremstyle{definition}
\newtheorem{definition}[equation]{Definition}

\newtheorem{remark}[equation]{Remark}

\newcommand{\C}{\mathbb{C}}

\newcommand{\LL}{\mathcal{L}}

\DeclareMathOperator{\gap}{gap}
\DeclareMathOperator{\charp}{char}

\newcommand{\iter}[1]{^{\langle #1\rangle}}

\usepackage[colorlinks,pagebackref,pdftex, bookmarks=false]{hyperref}
\title{Functional equations in polynomials}

\author{Zhan Jiang}
\address{
Zhan Jiang \\
  Department of Mathematics \\
  University of Michigan \\
  Ann Arbor, MI 48109--1043 \\
  USA
}
\email{zoeng@umich.edu}
\urladdr{\url{www-personal.umich.edu/~zoeng/}}

\author{Michael E. Zieve}
\address{
Michael Zieve \\
  Department of Mathematics \\
  University of Michigan \\
  Ann Arbor, MI 48109--1043 \\
  USA
}
\email{zieve@umich.edu}
\urladdr{\url{www.math.lsa.umich.edu/~zieve/}}

\date{}

\begin{document}

\begin{abstract}
We determine all $F,G\in\C[X]$ of degree at least $2$ for which the semigroup generated by $F$ and $G$
under composition is not the free semigroup on the letters $F$ and $G$.  We also solve the same problem
for $F,G\in X^2\C[[X]]$, and prove partial analogues over arbitrary fields.
\end{abstract}

\thanks{The authors thank Tien-Cuong Dinh for raising the question studied in this paper, and
Congling Qiu for observing that our proof for polynomials yields an analogous
result for power series.  The second author thanks the NSF for support under grants DMS-1162181 and DMS-1601844.}

\maketitle

%################################################################################
%################################################################################
%################################################################################

\section{Introduction}

Around 1920, Fatou, Julia and Ritt proved several remarkable results about functional equations.
These include, among other things, classifications of all pairs of complex polynomials $F,G\in\C[X]$
of degree at least $2$ which either commute or have a common power under the operation of functional
composition.  In other words, they solved each of the equations
\begin{equation} \label{commeq}
F \circ G = G \circ F
\end{equation}
and
\begin{equation} \label{itereq}
F\iter{r} = G\iter{s}\quad\text{ (for fixed $r,s>0$)}
\end{equation}
in polynomials $F,G\in\C[X]$ of degree at least $2$, where $F\iter{r}$ denotes the $r$-th iterate of $F$.
These results have played a fundamental role in many subsequent investigations, for instance
\cite{BD, BGT, BT, GN, GTZ1, GTZ2, GY, I, MS, N, Y}.

Both (\ref{commeq}) and (\ref{itereq}) express an equality between two compositions
of copies of $F$ and $G$, and it is natural  to study further equalities of this form.  In this paper we classify
all $F,G\in\C[X]$ of degree at least $2$ for which there is a nontrivial equality between two compositions of
copies of $F$ and $G$.    This answers a question posed by Tien-Cuong Dinh in 2009.
% Dinh asked this question to M.Z. in July 2009 at MPI in Bonn.

Our first conclusion is that if the semigroup generated by $F$ and $G$ under composition is not the free
semigroup on the letters $F$ and $G$,
then there is an integer $r>0$ for which  $F\iter{r} \circ G$  and  $F\iter{r}$  commute.
  Thus, if there is any functional relationship between $F$ and $G$ then there is a relationship
of the very simple type  $F\iter{r} \circ G \circ F\iter{r} = F\iter{2r} \circ G$.  We deduce this unexpected
conclusion more generally for certain classes of rational functions over an arbitrary field:
 
 \begin{thm} \label{ratfun}
Let $K$ be an algebraically closed field, let $m,n>1$ be integers not divisible by $\charp(K)$, and let 
$A,B,C,D\in K[X]$ be nonzero
polynomials such that $\deg(A)-\deg(B)=m$ and $\deg(C)-\deg(D)=n$.  If we put $F:=A/B$ and $G:=C/D$
then the following are equivalent:
\renewcommand{\theenumi}{\ref{ratfun}.\arabic{enumi}}
\renewcommand{\labelenumi}{(\thethm.\arabic{enumi})}
\begin{enumerate}
\item\label{ratfun1} the semigroup generated by $F$ and $G$ under composition is not the free semigroup on the
letters $F$ and $G$
\item\label{ratfun2} there is a root of unity $\gamma\in K^*$ for which the semigroup generated by $F$ and $G$ is
isomorphic to the semigroup generated by $X^m$ and $\gamma X^n$, via an isomorphism which maps $F$ and
$G$ to $X^m$ and $\gamma X^n$
\item\label{ratfun3} there exists $r>0$ for which  $F\iter{r} \circ G$  and  $F\iter{r}$  commute
\item\label{ratfun4} there exist $i\ge 0$ and $j>0$ such that, for every $r\ge i$ and every $s\ge 1$, the functions
  $F\iter{r} \circ G\iter{s}$  and  $F\iter{j}$  commute.
\end{enumerate}
\end{thm}
\renewcommand{\theenumi}{\arabic{enumi}}
\renewcommand{\labelenumi}{(\arabic{enumi})}

Since Ritt has classified the commuting complex rational functions \cite{Ritt1923}, we know a good deal
about the shape of  $F\iter{r} \circ G$  and  $F\iter{r}$  in \eqref{ratfun3} when $K=\C$.  However, due to the
non-explicit nature of Ritt's result, it is not clear whether this information makes it possible to classify all 
$F$ and $G$ satisfying \eqref{ratfun3}
when $K=\C$.  Ritt's result does take an explicit form when $F$ and $G$ are polynomials.
By combining the polynomial case of Ritt's result with several additional ingredients, we
obtain the following classification of the polynomials $F$ and $G$ satisfying \eqref{ratfun3}.  Here if 
$L(X)\in K[X]$ has degree $1$ then we write $L\iter{-1}$ for the inverse
of $L$ under composition.

\begin{thm} \label{poly}
Let $K$ be an algebraically closed field, and let $F,G \in K[X]$ have degrees $m,n$ where $m,n>1$ and
$\charp(K)\nmid mn$.  Then the following are equivalent:
\begin{enumerate}
\item[(a)] the semigroup generated by $F$ and $G$ under composition is not the free semigroup on the letters $F$
 and $G$
\item[(b)] there is a root of unity $\gamma\in K^*$ for which the semigroup generated by $F$ and $G$ is
isomorphic to the semigroup generated by $X^m$ and $\gamma X^n$, via an isomorphism which maps
$F$ and $G$ to $X^m$ and $\gamma X^n$
\item[(c)] there is a degree-one polynomial $L(X) \in K[X]$ for which $L\iter{-1} \circ F \circ L$  and  
$L\iter{-1} \circ G \circ L$ are one of the following:
\renewcommand{\theenumi}{}
\renewcommand{\labelenumi}{}
\renewcommand{\theenumii}{\ref{poly}.\arabic{enumii}}
\renewcommand{\labelenumii}{(\thethm.\arabic{enumii})}
\begin{enumerate}
\item\label{poly1} $X^m$ and $\alpha X^n$,  where  $\alpha\in K^*$ is a root of unity
\item\label{poly2} $\alpha T_m(X)$ and $\beta T_n(X)$, where  $\alpha,\beta\in\{1,-1\}$
\item\label{poly3} $\alpha H\iter{i}(X)$ and $\beta H\iter{j}(X)$, where $i,j>0$, $H(X)=X^r C(X^s)$ 
for some $C(X)\in K[X]$ and some $s>0\le r$, and $\alpha,\beta\in K^*$
 satisfy  $\alpha^s=\beta^s=1$.
\end{enumerate}
\end{enumerate}
\end{thm}
\renewcommand{\theenumii}{\arabic{enumi[}}
\renewcommand{\labelenumii}{(\arabic{enum[i})}
\renewcommand{\theenumi}{\arabic{enumi}}
\renewcommand{\labelenumi}{(\arabic{enumi})}

\noindent
In case \eqref{poly2}, $T_m(X)$ denotes the ``normalized Chebyshev polynomial", which is the unique polynomial in
$K[X]$ for which 
\[
T_m\Bigl(X+\frac1X\Bigr)=X^m+\frac1{X^m}.
\]
  We give more information about $T_m(X)$ in 
Remark~\ref{chebrem}.

\begin{remark}
The semigroups in (b) can be quite complicated.  For instance, if $\gamma\in K^*$ is a primitive
fifth root of unity then the semigroup generated by $F:=X^2$ and $G:=\gamma X^3$ has the following nine relations, none
of which is generated by the others:
\begin{align*}
G\circ F\iter{2}\circ G&=F\circ G\iter{2}\circ F \\
F\circ G\circ F\circ G&=G\iter{2}\circ F\iter{2} \\
F\circ G\iter{4}&=G\iter{4}\circ F \\
F\iter{2}\circ G\iter{3}&=G\iter{2}\circ F\circ G\circ F \\
F\iter{3}\circ G\iter{2}&=G\circ F\circ G\circ F\iter{2} \\
F\iter{4}\circ G&=G\circ F\iter{4} \\
F\circ G\circ F\iter{3}\circ G&=G\circ F\iter{3}\circ G\circ F \\
F\circ G\iter{3}\circ F\circ G&=G\circ F\circ G\iter{3}\circ F \\
F\iter{2}\circ G\iter{2}\circ F\circ G&=G\iter{3}\circ F\iter{3}.
\end{align*}
\end{remark}

Our proof of Theorem~\ref{poly} relies on the known classifications of commuting polynomials and of polynomials
with a common iterate.  In case $K=\C$, the latter classification was first proved by Ritt~\cite{Ritt1920}, using a
method which inspired the proof in the present paper.  This latter classification, together with a result proved
independently by Fatou~\cite{Fatou}, Julia~\cite{Julia} and Ritt~\cite{Ritt1923}, yields a classification of
commuting polynomials in $\C[X]$.
%It would be interesting if the methods of the present paper could be used to simplify the proof of
%the classification of commuting polynomials.  
Dorfer and Woracek~\cite{DW} generalized these results to arbitrary
fields of characteristic not dividing the degrees of the polynomials.  
We note that the proofs by Fatou, Julia and Ritt of the classification of commuting polynomials in $\C[X]$ do not 
extend easily to
positive characteristic, and in fact the proof for arbitrary fields in \cite{DW} yields a new proof in this classical case.
In addition to using two results from \cite{DW}, we also use methods adapted from that paper.

 We also solve the analogous problem for formal power series in $X^2 K[[X]]$ when $\charp(K)=0$, and give a
 partial solution when $\charp(K)>0$:

\begin{thm} \label{powerseries}
\renewcommand{\theenumi}{\ref{powerseries}.\arabic{enumi}}
\renewcommand{\labelenumi}{(\thethm.\arabic{enumi})}
Let $K$ be an algebraically closed field, and let $F,G \in K[[X]]$ have lowest-degree terms of degrees $m,n>1$,
where $\charp(K)\nmid mn$.  The following are equivalent:
\begin{enumerate}
\item\label{powerseries1} the semigroup generated by $F$ and $G$ under composition is not the free semigroup on the
 letters $F$ and $G$
\item\label{powerseries2} there is some $L \in  K[[X]]$  having lowest-degree term of degree $1$  for which
  $L\iter{-1} \circ F \circ L = X^m$  and  $L\iter{-1} \circ G \circ L = \alpha X^n$,  where  $\alpha\in K^*$  is a root
   of unity.
\end{enumerate}
\end{thm}
\renewcommand{\theenumi}{\arabic{enumi}}
\renewcommand{\labelenumi}{(\arabic{enumi})}

The arguments in this paper are elementary, given the known results that we use,
but we emphasize that the arguments are very subtle and need to be done in just the right way in order to work.  In particular, these arguments are quite difficult to come up with.

After we sent a preliminary version of this paper to two colleagues, our results have been used in two recent papers.  Bell, Bruin, Huang, Peng and Tucker use our results to prove a particularly strong form of the Tits alternative for polynomials \cite{JT}.  Also Hindes has extended our proof of Lemma~\ref{rtof1} to yield an analogous result for finite sets of polynomials \cite[Thm.~4.2]{H}.  The combination of this initial response and the many important applications of the earlier results of Fatou, Julia and Ritt makes us optimistic that our results will find many further applications.

%################################################################################
%################################################################################

\section{B\"ottcher functions}

In this section we show that certain elements of $K[[X]]$ are conjugate (under composition) to $X^m$.
This will play a crucial role in the proofs of our main results.

\begin{lemma} \label{Boettcherlemma}
Let $K$ be a field, and let $F(X) \in K[[X]]$ have lowest-degree term  $\alpha_m X^m$,  where
$\alpha_m\in K^*$, $\charp(K)\nmid m$, and $m>1$.
Then the following are equivalent:
\renewcommand{\theenumi}{\ref{Boettcherlemma}.\arabic{enumi}}
\renewcommand{\labelenumi}{(\thethm.\arabic{enumi})}
\begin{enumerate}
\item\label{Boett1}  there exists $L \in K[[X]]$ with lowest-degree term of degree $1$ such that  
$L\iter{-1} \circ F \circ L = X^m$
\item\label{Boett2} $\alpha_m$  is an $(m-1)$-th power in  $K$.
\end{enumerate}
Moreover, if $L_0$ is a series as in \eqref{Boett1} then all series $L$ as in \eqref{Boett1}
 have the form $L_0\circ \gamma X$
where $\gamma\in K$ satisfies $\gamma^{m-1}=1$.
\end{lemma}
\renewcommand{\theenumi}{\arabic{enumi}}
\renewcommand{\labelenumi}{(\arabic{enumi})}

\begin{proof}
Write  $F(X) = \sum_{i=m}^{\infty} \alpha_i X^i$ with $\alpha_i\in K$.
Consider an arbitrary  $L(X) = \sum_{i=1}^{\infty} \beta_j X^j$ with $\beta_1\ne 0$, where the $\beta_j$ are
elements of $K$ which we will determine.  The equation
\begin{equation} \label{B1}
L\iter{-1} \circ F \circ L = X^m
\end{equation}
is equivalent to
\begin{equation} \label{B2}
F \circ L = L \circ X^m.
\end{equation}
Both sides of (\ref{B2}) have lowest-degree term of degree $m$.  Equating coefficients of $X^m$ gives
$\alpha_m \beta_1^m = \beta_1$,
so that  $\beta_1^{m-1} = 1/\alpha_m$, whence \eqref{Boett1} implies \eqref{Boett2}.
For any  $\beta_1\in K$  satisfying $\beta_1^{m-1} = 1/\alpha_m$,  we successively compare the coefficients of
$X^{m+r}$ in both sides of (\ref{B2}) for $r=1,2,3,\dots$.
The coefficient of $X^{m+r}$ in  $F\circ L$  is the sum of
$\alpha_m m \beta_1^{m-1} \beta_{r+1}$ and some function of $F$ and the $\beta_j$'s with $j\le r$,
while the coefficient of $X^{m+r}$ in  $L \circ X^m$  is a function of $F$ and the $\beta_j$'s with $j\le r$.
Thus, for any choice of $\beta_j$'s with $2\le j\le r$, there is a unique choice of $\beta_{r+1}$ for which the two
sides of (\ref{B2}) have the same coefficient of $X^{m+r}$.  Inductively, it follows that there is a
unique $L\in XK[[X]]$ satisfying (\ref{B2}) for which $L'(0)$ is a fixed $(m-1)$-th root of $1/\alpha_m$.
Fix one such $L_0(X)$ and put $\beta_1:=L_0'(0)$, and note that by composing (\ref{B2}) on the right with $\gamma X$ 
where $\gamma^{m-1}=1$ we obtain $F\circ L_0(\gamma X) =L_0(\gamma^m X^m)= L_0(\gamma X) \circ X^m$,  so that  $L_0(\gamma X)$ 
 is the unique solution of (\ref{B2}) having lowest-degree term  $\beta_1 \gamma X$.  Therefore these solutions 
$L_0(\gamma X)$  comprise all solutions of (\ref{B1}).
\end{proof}

\begin{remark}
In case $K=\C$, the series $L\iter{-1}(X)$ satisfying \eqref{Boett1} are called 
\emph{B\"ottcher functions}, since B\"ottcher proved their existence in 1904~\cite[p.~176]{Boettcher}.  
In this case  $L(X)$  can be taken to be the unique analytic function defined in an open
neighborhood of $0$ such that the sequence of functions  $(F\iter{i}(X))^{1/m^i}$  converges uniformly (pointwise)
to  $L\iter{-1}(X)$,  where for each function we choose the  $m^i$-th root of  $F\iter{i}(X)$  which has derivative 
at $0$ equal to some fixed $(m-1)$-th root of $\alpha_m$.  
These B\"ottcher functions play a prominent role in complex dynamics, in connection with the dynamics near a
superattracting fixed point; for instance, see \cite[Chap.~9]{Milnor}.
\end{remark}

\begin{remark}
The hypotheses in Lemma~\ref{Boettcherlemma} cannot be removed.  For instance, if $\charp(K)=p>0$
then the equation  $L\iter{-1} \circ (X^p+X^{p+1}) \circ L = X^p$  has no solutions in $L\in X K^*+X^2 K[[X]]$,
since equating terms of degree $p+1$ in  $(X^p+X^{p+1}) \circ L = L \circ X^p$  gives  $L'(0)^{p+1} = 0$.
\begin{comment}
If the lowest-degree term of $F(X)$ has degree $1$ then  $F(X)$  is conjugate to  $F'(0)\cdot X$
if  $F'(0)$  is not a root of unity.  If the lowest-degree term of $F(X)$ has the form  $\alpha X$  where  $\alpha$  is a
root of unity, then some iterate of  $F(X)$  has lowest-degree term $X$; if $F(X)$ has infinite order under
 composition, and $\charp(K)=0$, then this iterate is conjugate to  $X+X^r+\beta X^{2r-1}$  for some $\beta\in K$
 and some $r>1$.
\end{comment}
\end{remark}

\section{Semigroups of formal power series}

In this section we prove Theorems~\ref{ratfun} and \ref{powerseries}.  By Lemma~\ref{Boettcherlemma},
it suffices to prove Theorem~\ref{powerseries} in case $F(X)=X^m$.  We first show that in this case 
$F(X)$ and $G(X)$ generate a free semigroup if $G(X)$ is not a monomial.

\begin{lemma} \label{power2}
Let $K$ be a field, let $m,n>1$ be integers not divisible by $\charp(K)$, and let $G\in K[[X]]$ have
lowest-degree term of degree $n$.  If  $G$  has at least two terms then the semigroup generated by
$X^m$ and $G$ under composition is the free semigroup on the letters $X^m$ and $G$.
\end{lemma}

\begin{proof}
For any nonzero $H \in K[[X]]$, let $\gap(H)$ be the (positive)
difference between the degrees of the two terms in $H$ of least degree, defining $\gap(H)$ to
be $\infty$ if $H$ has only one term.
If the lowest-degree term of $H$ has degree $r\ge 2$, and we denote $F:=X^m$, then plainly
\begin{itemize}
\item $\gap(H \circ F) = m\cdot\gap(H)$ 
\item $\gap(F \circ H) = \gap(H)$ 
\item $\gap(G \circ H) \ge  \min(\gap(H),\,r\cdot\gap(G))$, with equality holding if $\gap(H)\ne r\cdot\gap(G)$.
\end{itemize}
Thus if $B$ is a composition of copies of $F$ and $G$ then
$\gap(B \circ G) = \gap(G)$.
But if $A$ is a composition of copies of $F$ and $G$ then
$\gap(A \circ F)$  equals $\infty$ if there are no $G$'s in $A$, and otherwise equals
$m^i\cdot\gap(G)$  where  $i>0$ is the number of $F$'s in $A \circ F$ which occur to the right of every $G$.
(To see this, note that the lowest-degree term of  $G \circ F\iter{i}$  has degree  $n m^i$.)

If $F$ and $G$ do not generate the free semigroup on the letters $F$ and $G$ then there is a relation
$A \circ F = B \circ G$  where both $A$ and $B$ are compositions of copies of $F$ and $G$  (note that we may
assume that the rightmost letters on the two sides in the relation are distinct, since
$U \circ H = V \circ H$  implies  $U=V$  for any nonconstant  $H  \in  X K[[X]]$).  Thus
\[
\gap(G) = \gap(B \circ G) = \gap(A \circ F) \ge  m\cdot\gap(G), 
\]
which is impossible.
\end{proof}

Next we show that the semigroup generated by  $X^m$  and  $\alpha X^n$  is free if $\alpha$ is not a root of unity.

\begin{lemma} \label{rtof1}
Let  $K$  be a field, let $m,n>1$ be integers, and let $\alpha\in K^*$ have infinite order.
Then the semigroup generated by $F:=X^m$ and $G:=\alpha X^n$ under composition is the free semigroup on
$F$ and $G$.
\end{lemma}

\begin{proof}  Suppose otherwise.
If $U,V\in K[X]$ are compositions of copies of $F$ and $G$, and $H\in\{F,G\}$, then we must have $U=V$ if either $H\circ U=H\circ V$ or $U\circ H=V\circ H$.  Thus there exist $A,B\in K[X]$ which are compositions of copies of $F$ and $G$, and such that $A\circ F=B\circ G$ and the leftmost factor (namely $F$ or $G$) in $A\circ F$ is distinct from the leftmost factor in $B\circ G$.  When speaking of a monomial with leading coefficient $\alpha^s$, we refer to $s$ as the degree
of $\alpha$ in the monomial.  We can write $\{A \circ F,\, B \circ G\} = \{F \circ U,\, G \circ V\}$ where $U,V$ are 
compositions of copies of $F$ and $G$.  Then the degree of $\alpha$ in $F \circ U$ is divisible by $m$, but the 
degree of $\alpha$ in $G \circ V$ is congruent to $(1+n+n^2+\dots+n^{r-1}) \pmod{m}$,  where  $r\ge 1$ is the 
number of initial copies of $G$ in $G \circ V$ (counting from the left).  Hence $\gcd(m,n)=1$, % if not, there is some integer dividing the power of RHS but not LHS, which is a contradiction.
so by equating
degrees in the identity $A\circ F=B\circ G$ we see that $A \circ F$ and
$B \circ G$ involve the same number (say $u$) of $F$'s as one another, and also the same number (say $v$) of 
$G$'s.  The degree of $\alpha$ in  $A \circ F$ is at most the degree of $\alpha$ in  
$F\iter{u-1} \circ G\iter{v} \circ F$, which is
\begin{align*}
\sum_{i=0}^{v-1} m^{u-1} n^i =& m^{u-1} n^{v-1} + \sum_{i=0}^{v-2} m^{u-1} n^i \\
&= m^{u-1} n^{v-1} + m^{u-1} \frac{n^{v-1}-1}{n-1}  \\
&< m^{u-1} n^{v-1} + m^{u-1} n^{v-1} \\
&\le m^u n^{v-1},
\end{align*}
while the degree of $\alpha$ in $B \circ G$ is at least $m^u n^{v-1}$,  so these degrees cannot be the same.
\end{proof}

If $\alpha$ is an $\ell$-th root of unity then the semigroup generated by $X^m$ and $\alpha X^n$ is not free, since
every composition of copies of these polynomials is a monomial with leading coefficient a power of $\alpha$,
but for suitable degrees $d$ there are more than $\ell$ ways to compose copies of $X^m$ and $\alpha X^n$
in order to obtain a degree-$d$ polynomial.  We now describe some specific relations in this case.

\begin{lemma} \label{rtof1bis}
Let $K$ be a field, and define $F,G\in K[X]$ by $F:=X^m$ and $G:=\alpha X^n$ where $m,n>1$ and 
$\alpha\in K^*$ has finite order.  Then there exist $i\ge 0$ and $j>0$ such that, for every $r\ge i$ and every 
$s\ge 1$, the  polynomials $F\iter{r}\circ G\iter{s}$ and $F\iter{j}$ commute under composition.  In particular, 
there exists $r>0$ for which  $F\iter{r} \circ G$  and  $F\iter{r}$  commute.
\end{lemma}

\begin{proof}
Let $\ell$ be the multiplicative order of $\alpha$, and let $i\ge 0$ and $j>0$ be integers for which 
$\ell\mid m^i(m^j-1)$.
For any integers $r\ge i$ and $s\ge 1$, the polynomials
$F\iter{r} \circ G\iter{s}\circ F\iter{j}$  and  $F\iter{j+r} \circ G\iter{s}$
are monomials of the same degree with leading coefficients 
$\alpha^{m^r(1+n+n^2+\ldots+n^{s-1})}$  and  $\alpha^{m^{j+r}(1+n+n^2+\ldots+n^{s-1})}$, respectively.
Since  
\[
m^{j+r}-m^r = m^r(m^j-1) \equiv 0 \pmod{\ell},
\]
these leading coefficients are identical, so that
$F\iter{r} \circ G\iter{s} \circ F\iter{j} = F\iter{j+r} \circ G\iter{s}$, whence $F\iter{r}\circ G\iter{s}$ and $F\iter{j}$
commute.  Finally, if we let $r$ be any positive multiple of $j$ such that $r\ge i$, then $F\iter{r}\circ G$
commutes with $F\iter{j}$ and hence with $F\iter{r}$.
\end{proof}

We now prove Theorem~\ref{powerseries}.

\begin{proof}[Proof of Theorem~\ref{powerseries}]
Lemma~\ref{rtof1bis} shows that \eqref{powerseries2} implies \eqref{powerseries1}.
Conversely, if \eqref{powerseries1} holds then by Lemma~\ref{Boettcherlemma} there exists $L\in K[[X]]$ having
lowest-degree term of degree $1$ for which  $L\iter{-1} \circ F \circ L = X^m$.  Writing $A:=X^m$  and  
$B:=L\iter{-1} \circ G \circ L$, it follows that $A$ and $B$ do not generate the free semigroup on the letters
$A$ and $B$.  Since the lowest-degree term of  $B$ has degree $n$, Lemmas~\ref{power2} and \ref{rtof1} imply
that  $B=\alpha X^n$  where  $\alpha$  is a root of unity, so that \eqref{powerseries2} holds.
\end{proof}

We conclude this section by proving Theorem~\ref{ratfun}.

\begin{proof}[Proof of Theorem~\ref{ratfun}]
The implications \eqref{ratfun4}$\implies$\eqref{ratfun3}$\implies$\eqref{ratfun1} are immediate, and
\eqref{ratfun2} implies \eqref{ratfun4} by Lemma~\ref{rtof1bis}.  Thus we need only show that
\eqref{ratfun1} implies \eqref{ratfun2}.  Assume that \eqref{ratfun1} holds.
Put $U:=1/F(1/X)$ and $V:=1/G(1/X)$, so that $0$ is a root of both $U$ and $V$, with respective multiplicities
$m$ and $n$.  By viewing $K(X)$ as a subfield of $K((X))$ in the usual manner,
it follows that $U$ and $V$ are elements of $K[[X]]$ whose lowest-degree terms have degrees $m$ and $n$.
Since \eqref{ratfun1} holds, the semigroup generated by $U$ and $V$
is not the free semigroup on the letters $U$ and $V$.  By Theorem~\ref{powerseries}, there exists 
$L(X)\in  K[[X]]$ with lowest-degree term of degree $1$ for which
$R:=L\iter{-1}\circ U\circ L$ and $S:=L\iter{-1}\circ V\circ L$ satisfy $R=X^m$ and $S=\alpha X^n$
with $\alpha\in K^*$ having finite order.  Now let $\psi$ be the automorphism of the semigroup 
$K(X)\setminus K$ (under the operation of composition) defined by $\psi(W)=1/W(1/X)$, and let
$\phi$ be the automorphism of the semigroup $XK[[X]]$ (again under composition) defined by
$\phi(W)=L\iter{-1}\circ W\circ L$.  Writing $\langle F,G\rangle$ for the semigroup generated by $F$ and $G$,
the restriction of $\psi$ to $\langle F,G\rangle$ has image contained in $X^2 K[[X]]$, so that the composition
$\phi\circ\psi$ induces an isomorphism from $\langle F,G\rangle$ to the semigroup generated by
$\phi(\psi(F))=X^m$ and $\phi(\psi(G))=\alpha X^n$.
\end{proof}
 
\begin{remark}
The conclusion of Theorem~\ref{ratfun} holds for any $F,G\in K(X)$ for which there is a degree-one $M\in K(X)$
for which $\widehat F:=M\iter{-1}\circ F\circ M$ and $\widehat G:=M\iter{-1}\circ G\circ M$ satisfy the hypotheses of
Theorem~\ref{ratfun}.  In other words, this conclusion holds for any nonconstant $F,G\in K(X)$ which have
a common fixed point $\beta\in\mathbb{P}^1(K)$ whose ramification indices $m,n$ in the covers
$F,G\colon\mathbb{P}^1\to\mathbb{P}^1$ satisfy $m,n>1$ and $\charp(K)\nmid mn$.
\end{remark} 
 
\begin{remark}
It would be interesting to prove versions of Theorems~\ref{ratfun} and \ref{powerseries} in case 
$\charp(K)\mid mn$.
Although the conclusion of Lemma~\ref{Boettcherlemma} is not true in this setting, Ruggiero~\cite{Ruggiero}
has provided a potential replacement for Lemma~\ref{Boettcherlemma}, by giving a somewhat
manageable set of representatives for $X^2 K[[X]]$ up to conjugation by elements of  $X K^* + X^2 K[[X]]$.
A thorough study of Ruggiero's representatives would be valuable for both this and other questions.
\end{remark}

%################################################################################
%################################################################################

\section{Semigroups of polynomials}

In this section we prove Theorem~\ref{poly}.  We will need several preliminary results.

\begin{definition}
We say that a polynomial of degree $m>0$ has \emph{gap form} if it has no term of degree $m-1$.
\end{definition}

\begin{lemma} \label{gap}
Let $K$ be a field and let $F,G\in K[X]$ have degrees $m,n>0$ where $\charp(K)\nmid m$.
Suppose that either $n>1$ or $F$ has gap form.  Then  $F \circ G$  has gap form  if and only if  $G$  has gap form.
\end{lemma}

\begin{proof}
Write $F=\sum_{i=0}^{m} \alpha_i X^{m-i}$ and $G=\sum_{j=0}^{n} \beta_j X^{n-j}$ where $\alpha_i,\beta_j\in K$.
Then the coefficient of $X^{mn-1}$ in $F \circ G$ is $\alpha_0 m \beta_0^{n-1} \beta_1$, which equals $0$ 
precisely when $\beta_1=0$.
\end{proof}

\begin{definition}
For any field $K$ and any  $F(X) \in K[X]$ which has at least two terms, we define $\LL(F)$ to be the
greatest common divisor of the set of all differences between degrees of pairs of terms of $f(X)$.  Thus, $\LL(F)$
is the largest integer $s$ for which we can write  $F(X) = X^r A(X^s)$  with  $r\ge 0$  and  $A\in K[X]$.  If  $F(X)$  is
a monomial then we define $\LL(F):=0$.
\end{definition}

\begin{lemma} \label{parity}
For any field $K$, any nonconstant $F(X) \in K[X]$, and any $\alpha\in K^*$, the following are equivalent:
\begin{itemize}
\item there exists $\beta\in K$ for which $\beta F(X) = F(\alpha X)$
\item $\alpha^{\LL(F)}=1$.
\end{itemize}
\end{lemma}

\begin{proof}
Compare coefficients.
\end{proof}

\begin{lemma} \label{Iiterate}
Let $K$ be a field, and let $F(X)\in K[X]$ have degree $m>1$ where $\charp(K)\nmid m$.  
Then  $\LL(F\iter{i})=\LL(F)$  for every  $i>0$.
\end{lemma}

\begin{proof}
Let $G(X)\in K[X]$ have degree $n>0$.  We will show that
\begin{enumerate}
\item $\gcd(\LL(F),\LL(G)) \mid  \LL(F\circ G)$  if $\LL(F)+\LL(G)>0$
\item If $F$ has gap form then $\LL(F \circ G)\mid \LL(G)$.
\end{enumerate}
First we show that the conclusion of Lemma~\ref{Iiterate} follows from (1) and (2).  
If $F$ has gap form and $\LL(F)>0$ then (1) implies that $\LL(F)\mid\LL(F\iter{i})$,
and (2) implies that $\LL(F\iter{j+1})\mid \LL(F\iter{j})$ for every $j>0$, so that 
$\LL(F\iter{i})\mid \LL(F)$ and thus $\LL(F)=\LL(F\iter{i})$.
If $\LL(F)=0$ then $F$ is a monomial, so also $F\iter{i}$ is a monomial and thus $\LL(F\iter{i})=0$.
Finally, if $F$ does not have gap form then Lemma~\ref{gap} implies that $F\iter{i}$ does not have gap form,
so that $\LL(F\iter{i})=1=\LL(F)$.

It remains to prove (1) and (2).  We begin with (1).  Let $d:=\gcd(\LL(F),\LL(G))$, so that $F(X)=X^r A(X^d)$ and
$G(X)=X^s B(X^d)$ for some $A,B\in K[X]$ and some $r,s\ge 0$.  Then plainly $F\circ G$ equals $X^{rs} C(X^d)$
for some $C\in K[X]$, so that $d\mid \LL(F\circ G)$.  

Now we prove (2).  It suffices to prove that, if an integer $s>0$ does not divide $\LL(G)$, then
also $s\nmid \LL(F\circ G)$.  So let $s$ be a positive integer for which $s\nmid \LL(G)$.
Writing $G(X)=\sum_{j=0}^n \beta_j X^{n-j}$ with $\beta_j\in K$, it follows that $\beta_j\ne 0$ for some $j$ with
$s\nmid j$.  Letting $t$ be the least such integer $j$, we will show that $F\circ G$ has a term of degree $mn-t$,
which yields the desired conclusion since then $\LL(F\circ G)\mid t$ and thus $s\nmid \LL(F\circ G)$.
Let $\alpha$ be the leading coefficient of $F(X)$.  Since $F(X)$ has gap form, we know that 
$\deg(F(X)-\alpha X^m)\le m-2$, so that $(F(X)-\alpha X^m)\circ G(X)$ has degree at most $(m-2)n$.  
Since $t\le n<2n$,
it follows that the coefficient of $X^{mn-t}$ in $F\circ G$ equals the coefficient of $X^{mn-t}$ in $\alpha X^m\circ G$.
The latter coefficient is the sum of $\alpha m\beta_0^{m-1}\beta_t$ and
the coefficient of $X^{mn-t}$ in $\alpha X^m\circ H$, where $H(X):=\sum_{j=0}^{t-1} \beta_j X^{n-j}$.
Since all terms of $H(X)$ have degree congruent to $n\pmod{s}$, all terms of $H(X)^m$ must have
degree congruent to $mn\pmod{s}$, and thus $H(X)^m$ has no term of degree $mn-t$.  It follows that the
coefficient of $X^{mn-t}$ in $F\circ G$ equals $\alpha m\beta_0^{m-1}\beta_t$, and hence is nonzero.  Thus
$\LL(F\circ G)\mid t$, so that $s\nmid \LL(F\circ G)$.  Hence every positive integer which divides $\LL(F\circ G)$
must also divide $\LL(G)$, so we conclude that $\LL(F\circ G)\mid \LL(G)$.
\end{proof}

\begin{remark}
The hypothesis $\deg(F)>1$ cannot be removed, since for $F:=1-X$ we have $\LL(F)=1$ and $\LL(F\iter{2})=0$.
The hypothesis $\charp(K)\nmid \deg(F)$ cannot be removed, since for $F:=X^p+X$ with $p:=\charp(K)>0$ we have
$F\iter{p}=X^{p^p}+X$, so that $\LL(F\iter{p})=p^p-1>p-1=\LL(F)$.  In item (2) in the proof, the hypothesis that $F$
has gap form cannot be removed, in light of the example $F(X):=(X+1)^m$ and $G(X):=X^n-1$.
\end{remark}

We will also use two results of Dorfer and Woracek~\cite{DW}, which generalize classical results of Fatou,
Julia and Ritt.  We state these after explaining the polynomials $T_m(X)$ which occur in these results.

\begin{remark} \label{chebrem}
For any field $K$ and any positive integer $m$, we define the ``normalized Chebyshev polynomial" $T_m(X)$ 
to be the unique polynomial in $K[X]$ which satisfies $T_m(X+X^{-1})=X^m+X^{-m}$.  One way to construct
$T_m(X)$ is via the identity $T_m(X)=D(X,1)$, where $D(X,Y)\in K[X,Y]$ satisfies $D(X+Y,XY)=X^m+Y^m$.
Alternately, if $C_m(X)\in\mathbb{Z}[X]$
denotes the classical degree-$m$ Chebyshev polynomial which satisfies
$C_m(\cos\theta)=\cos m\theta$, then for $\charp(K)=0$ we have $T_m(X) = 2X\circ C_m(X)\circ X/2$.
The main advantage of $T_m(X)$ to $C_m(X)$ occurs in characteristic $2$, since the reduction mod~$2$ of 
$C_m(X)$ is either $1$ or $X$, whereas $T_m(X)$ is a monic degree-$m$ polynomial in $K[X]$ for every field $K$.
For further details and results about $T_m(X)$, see for instance \cite{ACZ,LMT}.
\end{remark}

\begin{proposition}\emph{(\cite[Thm.~5.1]{DW})} \label{commute}
Let $K$ be an algebraically closed field, let $m,n>1$ be integers
not divisible by $\charp(K)$, and let $F,G\in K[X]$ satisfy $\deg(F)=m$, $\deg(G)=n$, and $F \circ G = G \circ F$.
Then there is a degree-one $L\in K[X]$ for which the conjugates  $L\iter{-1} \circ F \circ L$  and  
$L\iter{-1} \circ G \circ L$  are one of the following:
\renewcommand{\theenumi}{\ref{commute}.\arabic{enumi}}
\renewcommand{\labelenumi}{(\thethm.\arabic{enumi})}
\begin{enumerate}
\item\label{commute1} $\alpha X^m$  and  $\beta X^n$,  where  $\alpha,\beta\in K^*$
\item\label{commute2} $\alpha T_m$  and  $\beta T_n$,  where  $\alpha,\beta\in\{1,-1\}$
\item\label{commute3} $\alpha H\iter{i}$  and  $\beta H\iter{j}$,  where $i,j>0$ and $\alpha,\beta\in K^*$ and $H\in K[X]$ 
satisfy  $\alpha^{\LL(H)}=\beta^{\LL(H)}=1$.
\end{enumerate}
\end{proposition}
\renewcommand{\theenumi}{\arabic{enumi}}
\renewcommand{\labelenumi}{(\arabic{enumi})}

\begin{proposition}\emph{(\cite[Thm.~4.3]{DW})} \label{commoniterate}
Let $K$ be an algebraically closed field, let $m,n>1$ be integers
not divisible by $\charp(K)$, and let $F,H\in K[X]$ have degrees $m$ and $n$, where  $H$  has gap form.
If $F\iter{r} = \beta H\iter{j}$  for some  $r,j>0$  and some  $\beta\in K^*$  such that  
$\beta^{\LL(F)}=\beta^{\LL(H)}=1$,
then there exists a polynomial $A(X) \in K[X]$ having gap form for which 
$\LL(F)\mid \LL(A)$ and  $F = \gamma A\iter{s}$  and  $H = \delta A\iter{t}$  for some $s,t>0$ and some
$\gamma,\delta\in K^*$ such that  $\gamma^{\LL(F)}=\delta^{\LL(F)}=1$.
\end{proposition}

\begin{remark} \label{sameI}
In light of Lemma~\ref{Iiterate}, the hypothesis $\beta^{\LL(F)}=\beta^{\LL(H)}=1$ in 
Proposition~\ref{commoniterate} can be weakened to either $\beta^{\LL(F)}=1$ or $\beta^{\LL(H)}=1$.  For, since
$F\iter{r}=\beta H\iter{j}$, we have
\[
\LL(F)=\LL(F\iter{r})=\LL(\beta H\iter{j})=\LL(H\iter{j})=\LL(H).
\]
\end{remark}

When applying the previous two results, we will need to know the polynomials which have an
iterate being either a monomial or a Chebyshev polynomial.  We describe these now.

\begin{lemma} \label{monit}
Let $K$ be a field and let $F\in K[X]$ satisfy $\deg(F)>1$ and $\charp(K)\nmid\deg(F)$.
If $F\iter{i}(X)$ is a monomial for some $i>0$ then $F$ is a monomial.
\end{lemma}

\begin{proof}
We have $(F\iter{i})^{-1}(0)=\{0\}$, so since  $F\iter{i} = F \circ F\iter{i-1}$  we see that $F^{-1}(0)$ consists
of a single element $\alpha$, which must be a root of the derivative  $F'(X)$  which in turn is a factor of 
$(F\iter{i})'(X)$.   Since  $(F\iter{i})'(X)$  is a monomial, it follows that  $\alpha=0$,  whence  $F$  is a monomial.
\end{proof}

\begin{comment}
\begin{remark}
The hypothesis $\charp(K)\nmid\deg(F)$ may be replaced by the weaker hypothesis that $F$ cannot be written
as  $G(X) \circ X^q$  for any degree-one  $G\in K[X]$  and any  $q$  which is a power of $\charp(K)$.
\end{remark}
\end{comment}

Our proof of the analogous assertion for Chebyshev polynomials uses the following simple but unexpected result.

\begin{lemma} \label{uniq}
If $K$ is a field and $A,B,F,G\in K[X]$ are nonconstant polynomials such that  $A \circ B = F \circ G$,
$\deg(A)=\deg(F)$, and $\charp(K)\nmid \deg(A)$, then there is a degree-one $L\in K[X]$ for which 
$F=A\circ L\iter{-1}$ and $G = L \circ B$.
\end{lemma}

Lemma~\ref{uniq} was first proved for $K=\C$ by Ritt \cite{Ritt1922}, via Riemann surface techniques, and was
later proved by Levi \cite[\S 2]{Levi} by comparing coefficients.
Although Levi stated the result only for fields of characteristic $0$, his proof works in arbitrary characteristic, as noted in
\cite[Prop.~2.2]{Turnwald}.
\begin{comment}
Lemma~\ref{uniq} can also be proved by considering Taylor expansions \cite[Sec. 4]{WZ}, or by working with formal
power series \cite[Cor.~2]{McConnell} or with inertia groups \cite[Cor.~2.9]{ZM}; in fact all three of
these proofs are stated in characteristic $0$, but the first immediately extends to positive characteristic
upon writing Taylor's formula in terms of Hasse derivatives, and the other two can be extended with some
effort.
\end{comment}

\begin{comment}
\begin{remark}
The hypothesis $\charp(K)\nmid \deg(f)$ cannot be removed, since for instance 
$(X^p-X) \circ (X+1) = (X^p-X) \circ X$
if $\charp(K)=p>0$.
\end{remark}
\end{comment}

Note that the definition of Chebyshev polynomials implies that  $T_i \circ T_j = T_{ij}$ and 
$T_m(-X)=(-1)^m T_m(X)$ (so that $T_m$ has gap form).   Also, if $m>1$ and $\charp(K)\nmid m$ then the two 
highest-degree terms of $T_m$ are $X^m$ and $-mX^{m-2}$, so that $\LL(T_m)=2$.

\begin{lemma} \label{chebit}
Let $K$ be a field, and let $F(X)\in K[X]$.  If $F\iter{i}(X)=\alpha T_n(X)$ for some $\alpha\in K^*$ and some
$i,n>1$ with $\charp(K)\nmid n$, then $\alpha\in\{1,-1\}$ and
$F(X) = \beta T_m(X)$ for some $\beta\in\{1,-1\}$ and some $m>1$.
\end{lemma}

\begin{proof}
For $m:=\deg(F)$ and any $j$ with $1\le j\le i$ we have
\[
\alpha T_{n/m^j} \circ T_{m^j} = \alpha T_n = F\iter{i} = F\iter{i-j} \circ F\iter{j},
\]
so by Lemma~\ref{uniq} there is a degree-one $L_j\in K[X]$ for which  $F\iter{j} =  L_j\circ T_{m^j}$ and $\alpha T_{n/m^j}=F\iter{i-j}\circ L_j$.
Thus $F\iter{j}$ has gap form, so Lemma~\ref{gap} implies that $F$ has gap form and thus
 $F\iter{i-j}$ has gap form.  Since also $T_{n/m^j}$ has gap form, and $\alpha T_{n/m^j}=F\iter{i-j}\circ L_j$, Lemma~\ref{gap} implies that $L_j$ has gap form, 
whence $L_j(0)=0$.  Substituting $F=L_1\circ T_m$ into $F\iter{2}=L_2\circ T_{m^2}$ yields
\[
L_1\circ T_m\circ L_1\circ T_m = L_2\circ T_m\circ T_m,
\]
so that
\[
L_1\circ T_m\circ L_1 = L_2\circ T_m.
\]
Here $L_1(X)=\beta X$, and Lemma~\ref{parity} implies that $\beta^{\LL(T_m)}=1$.
Since $\LL(T_m)=2$, we conclude that $\beta\in\{1,-1\}$, where $F=\beta T_m$.
Finally, comparing leading coefficients in $\alpha T_n(X)=F\iter{i}(X)$ yields $\alpha\in\{1,-1\}$.
\end{proof}

We now use the above results to prove Theorem~\ref{poly}.

\begin{proof}[Proof of Theorem~\ref{poly}]
By Theorem~\ref{ratfun} with $B=D=1$, we know that (a) is equivalent to (b).
We first show that (c) implies (a).  
If $F$ and $G$ are as in \eqref{poly1} then (b) and hence (a) holds.
Next consider $F:=\alpha T_m$
and $G:=\beta T_n$ with $\alpha,\beta\in\{1,-1\}$. 
If $m$ is even then $F\iter{2} \circ G = F \circ G \circ F$; likewise if $n$ is even then 
$G\iter{2}\circ F=G\circ F\circ G$, and finally if both $m$ and $n$ are odd then $F \circ G = G \circ F$.
Next, if $F$ and $G$ are as in \eqref{poly3} then
 $F\iter{j} = \gamma G\iter{i}$ for some $\gamma\in K^*$
with $\gamma^s=1$.  Let $a,b$ be  positive integers for which $r^{ia}(1+r^{ij}+r^{2ij}+\dots+r^{(b-1)ij})$
is divisible by $s$.  Then
\begin{align*}
F\iter{a+jb} &= F\iter{a} \circ \gamma^{1+r^{ij}+r^{2ij}+\dots +r^{(b-1)ij}} G\iter{ib} \\
&= \gamma^{r^{ia}(1+r^{ij}+r^{2ij}+\dots+r^{(b-1)ij})}  F\iter{a} \circ G\iter{ib} \\
&= F\iter{a} \circ G\iter{ib}.
\end{align*}

It remains to prove that (a) implies (c).  Thus we assume now that the semigroup generated by $F$ and $G$
 under composition is not the
free semigroup on the letters $F$ and $G$.   By Theorem~\ref{ratfun} with $B=D=1$,
we know that  $F\iter{r} \circ G$ commutes with  $F\iter{r}$  for some $r>0$.  By Proposition~\ref{commute},
there is a degree-one $L\in K[X]$ for which the conjugates $L\iter{-1} \circ F\iter{r} \circ G \circ L$ and 
$L\iter{-1} \circ F\iter{r} \circ L$ are one of \eqref{commute1},
 \eqref{commute2} or \eqref{commute3}.  Replace  $F$  and  $G$
by  $L\iter{-1} \circ F \circ L$  and  $L\iter{-1} \circ G \circ L$,  so that  $F\iter{r} \circ G$  and  $F\iter{r}$
  are one of \eqref{commute1}, \eqref{commute2} or \eqref{commute3};
this replacement does not affect whether or not the conclusion of (c) in Theorem~\ref{poly} is true.

First assume that  $F\iter{r} \circ G = \alpha X^m$  and  $F\iter{r} = \beta X^n$  where  $\alpha,\beta\in K^*$.
Then Lemma~\ref{monit} implies that  $F$  is a monomial, so that also $G$ is a monomial.  By replacing $F$ and 
$G$ by their conjugates by $\gamma X$ for a suitable $\gamma\in K^*$, we may assume that $F$ is monic. 
 Finally, Lemma~\ref{rtof1} implies that the leading coefficient of $G$ is a root of unity, so $F$ and $G$ are as in
  \eqref{poly1}.

Next assume that  $F\iter{r} \circ G = \alpha T_m$  and  $F\iter{r} = \beta T_n$  where  $\alpha,\beta\in\{1,-1\}$.
By Lemma~\ref{chebit} we have $F=\delta T_j$ for some $\delta\in\{1,-1\}$.  We compute
\[
\alpha T_n\circ T_{m/n} = F\iter{r}\circ G = \beta T_n\circ G,
\]
so by Lemma~\ref{uniq} there is a degree-one $L\in K[X]$ for which $G=L\circ T_{m/n}$ and
$\alpha T_n = \beta T_n\circ L$.  Since $T_n$ has gap form, this last identity implies that $L$ has gap form
by Lemma~\ref{gap}.  Thus from Lemma~\ref{parity} and the fact that $\LL(T_n)=2$ we conclude that
$L(X)=\pm X$, so that $G=\pm T_{m/n}$.  Therefore $F$ and $G$ are as in~\eqref{poly2}.

Finally, assume that  $F\iter{r} \circ G = \alpha H\iter{i}$  and  $F\iter{r} = \beta H\iter{j}$  where  $i,j>0$ and  
$\alpha,\beta\in K^*$ satisfy  $\alpha^{\LL(H)}=\beta^{\LL(H)}=1$ with $H\in K[X]$.
If $\LL(H)\ne 1$ then $H$ has gap form; if $\LL(H)=1$ then $\alpha=\beta=1$, and we may conjugate $F,G,H$ by a 
translation in order to put $H$ into gap form, where such conjugation preserves the two relations  
$F\iter{r} \circ G = H\iter{i}$ and  $F\iter{r} = H\iter{j}$.  Thus in any case we may assume that  $H$  is in gap form,
so by Proposition~\ref{commoniterate} and Remark~\ref{sameI}, the equation  $F\iter{r} = \beta H\iter{j}$
 implies that  $\LL(F)=\LL(H)$  and that there is a polynomial $A(X) \in K[X]$ which has gap form and satisfies  
 $\LL(F)\mid \LL(A)$  and  $F = \gamma A\iter{s}$  and  $H = \delta A\iter{t}$  for some $s,t>0$
and some $\gamma,\delta\in K^*$ such that
$\gamma^{\LL(F)}=\delta^{\LL(F)}=1$.  Substituting into the equation  $F\iter{r} \circ G = \alpha H\iter{i}$  yields  
$A\iter{rs} \circ G = \epsilon A\iter{it}$  where
$\epsilon^{\LL(F)}=1$.  Write this as  $A\iter{rs} \circ G = \epsilon A\iter{rs} \circ A\iter{it-rs}$, and use
Lemma~\ref{uniq} to conclude that $G = \eta + \zeta A\iter{it-rs}$  for some $\eta\in K$ and $\zeta\in K^*$, 
so that  $A\iter{rs} \circ (\zeta X+\eta) = \epsilon A\iter{rs}$.
Since $A$ has gap form, also $A\iter{rs}$ has gap form, so we must have $\eta=0$.  Lemma~\ref{parity} implies
 that  $\zeta^{\LL(A\iter{rs})}=1$, so $\zeta^{\LL(A)}=1$ by Lemma~\ref{Iiterate}.
Therefore  $G = \zeta A\iter{it-rs}$  and  $F =\gamma A\iter{s}$  where  $\zeta^{\LL(A)}=\gamma^{\LL(A)}=1$.
If $\LL(A)>0$ then $F$ and $G$ are as in \eqref{poly3}.
Finally, if $\LL(A)=0$ then we may replace $F$ and $G$ by their conjugates by $\theta X$ for a suitable
$\theta\in K^*$, in order to assume that $F=X^m$ and $G=\lambda X^n$ with $\lambda\in K^*$.
Lemma~\ref{rtof1} implies that $\lambda$ is a root of unity, so that $F$ and $G$ are as in \eqref{poly1}.
\end{proof}

%################################################################################
%################################################################################
%################################################################################

\end{document}